\documentclass[a4paper,12pt]{article}
\usepackage{times}
\usepackage[english]{babel}
\usepackage{amssymb,amsmath}
\numberwithin{equation}{section}
\usepackage{amsthm}
\usepackage{array}
\usepackage{wasysym}
\usepackage[noadjust]{cite}
\usepackage{hyperref}
\usepackage[height=22.7cm , width = 16cm , top = 4cm , left = 3cm, a4paper]{geometry}
\usepackage{amssymb}
\usepackage{amsmath}
\usepackage{cite}
 \usepackage{pstricks}
\usepackage{epsfig}
\usepackage{verbatim}
\usepackage{multicol}
\usepackage{graphicx, color, psfrag}
\usepackage{graphics, psfrag}
\usepackage[numbers,sort]{natbib}
 \newtheorem{theorem}{Theorem}[section]

\theoremstyle{definition}

\newcommand{\e}{\end{document}}

\begin{document}

\thispagestyle{empty}

\author{
{{\bf  M. A. El-Damcese$^1$, Abdelfattah Mustafa$^{2,}$\footnote{Corresponding author: abdelfatah\_mustafa@yahoo.com}, }} \\
{{ \bf B. S. El-Desouky$^2$ \; and   M. E. Mustafa$^2$} }
 { }\vspace{.2cm}\\
 \small \it $^{1}$Tanta University, Faculty of Science, Mathematics Department, Egypt.\\
 \small \it $^{2}$Department of Mathematics, Faculty of Science, Mansoura University, Mansoura 35516, Egypt.
}

\title{The Odd Generalized Exponential Gompertz Distribution}

\date{}

\maketitle
\small \pagestyle{myheadings}
        \markboth{{\scriptsize The Odd Generalized Exponential Gompertz Distribution }}
        {{\scriptsize {M. A. El-Damcese, Abdelfattah Mustafa, B. S. El-Desouky and   M.E. Mustafa}}}

\hrule \vskip 8pt

\begin{abstract}
In this paper we propose a new lifetime model, called the odd generalized exponential gompertz distribution, We obtained some of its mathematical properties. Some structural properties of the new distribution are studied. The method of maximum likelihood is used for estimating the model parameters and the observed Fisher's information matrix is derived. We illustrate the usefulness of the proposed model by applications to real data.
\end{abstract}



\section{Introduction}
In the analysis of lifetime data we can use the Gompertz, exponential and generalized exponential distributions. It is known that the exponential distribution have only constant hazard rate function where as Gompertz, and generalized exponential distributions can have only monotone (increasing in case of Gompertz and increasing or decreasing in case of generalized exponential distribution) hazard rate. These distributions are used for modelling the lifetimes of components of physical systems and the organisms of biological populations. The Gompertz distribution has received considerable attention from demographers and actuaries. Pollard and Valkovics \cite{14} were the first to study the Gompertz distribution, they both defined the moment generating function of the Gompertz distribution in terms of the incomplete or complete gamma function and their results are either approximate or left in an integral form. Later, Marshall and Olkin \cite{10} described the negative Gompertz distribution; a Gompertz distribution with a negative rate of aging parameter.

\noindent
Recently, a generalization of the Gompertz distribution based on the idea given in \cite{2} was proposed by \cite{3} this new distribution is known as generalized Gompertz (GG) distribution which includes the exponential (E), generalized exponential (GE), and Gompertz (G) distributions. A new generalization of th Gompertz (G) distribution which results of the application of the Gompertz distribution to the Beta generator proposed by  \cite{4}, called the Beta-Gompertz (BG) distribution which introduced by  \cite{8}. On the other hand the two-parameter exponentiated exponential or generalized exponential distribution (GE) introduced by \cite{2}. This distribution is a particular member of the exponentiated Weibull (EW) distribution introduced by \cite{11}. The GE distribution is a right skewed unimodal distribution, the density function and hazard function of the exponentiated exponential distribution are quite similar to the density function and hazard function of the Gamma distribution. Its applications have been wide-spread as model to power system equipment, rainfall data, software reliability and analysis of animal behavior.

\noindent
Recently  \cite{15} proposed a new class of univariate distributions called the odd generalized exponential (OGE) family and studied each of the OGE-Weibull (OGE-W) distribution, the OGE-Fr\'{e}chet (OGE-Fr) distribution and the OGE-Normal (OGE-N) distribution. This method is flexible because of the hazard rate shapes could be increasing, decreasing, bathtub and upside down bathtub.

\noindent
In this article we present a new distribution from the exponentiated exponential distribution and gompertz distribution called the Odd Generalized Exponential-Gompertz (OGE-G) distribution using new family of univariate distributions proposed by \cite{15}.

\noindent
A random variable $X$ is said to have generalized exponential (GE) distribution with parameters $\alpha ,\beta $ if the cumulative distribution function (CDF) is given by
\begin{equation} \label{1}
F(x)=\left( 1-e^{-\alpha x}\right)^{\beta }, \; x>0, \, \alpha>0, \beta >0.
\end{equation}

\noindent
The Odd Generalized Exponential family by \cite{15} is defined as follows. Let  $G(x; \epsilon) $ is the CDF of any distribution depends on parameter $\epsilon$ and thus the survival function is $\overline{G}(x,\epsilon )=1-G(x; \epsilon)$, then the CDF of OGE-family  is defined by replacing $x$ in CDF of GE in equation (\ref{1}) by $\frac{G(x;\epsilon )}{\overline{G}(x,\epsilon )}$ to get
\begin{equation} \label{2}
F(x;\alpha ,\epsilon ,\beta )=\left[ 1-e^{-\alpha \frac{G(x;\epsilon)}{\overline{G}(x,\epsilon )}}\right] ^{\beta}, \;x>0, \alpha >0,\epsilon >0, \beta>0.
\end{equation}

\noindent
This paper is outlined as follows. In Section 2, we define the cumulative distribution function, density function, reliability function and hazard function of the Odd Generalized exponential-Gompertz (OGE-G) distribution. In Section 3, we introduce the statistical properties include, the quantile function, the mode, the median and the moments. Section 4 discusses the distribution of the order statistics for (OGE-G) distribution. Moreover, maximum likelihood estimation of the parameters is determined in Section 5. Finally, an application of OGE-G using a real data set is presented in Section 6.

\section{The OGE-G Distribution}
\subsection{OGE-G specifications}
In this section we define new four parameters distribution called Odd Generalized Exponential-Gompertz distribution with parameters $\alpha,\lambda, c,\beta $ written as OGE-G($\Theta $), where the vector $\Theta $ is defined by $\Theta =(\alpha,\lambda, c,\beta )$.

\noindent
A random variable $X$ is said to have OGE-G with parameters $\alpha,\lambda, c,\beta $ if its cumulative distribution function given as follows
\begin{equation} \label{3}
F(x;\Theta )=\left\{ 1-e^{-\alpha \left[ e^{\frac{\lambda}{c}(e^{cx}-1)}-1\right] }\right\}^{\beta}, \;x>0,\alpha, \lambda, c, \beta>0,
\end{equation}

\noindent
where $\alpha,\lambda, c$  are scale parameters and $\beta $ is shape parameter.

\subsection{PDF and hazard rate}
If a random variable $X$ has CDF in (\ref{3}) then the corresponding probability density function is
\begin{equation} \label{4}
f(x;\Theta )=\alpha \beta \lambda e^{cx} e^{\frac{\lambda }{c}(e^{cx}-1)} e^{-\alpha \left[ e^{\frac{\lambda}{c}(e^{cx}-1)}-1\right]}
\left\{ 1-e^{-\alpha \left[e^{\frac{\lambda }{c}(e^{cx}-1)}-1\right] }\right\}^{\beta -1}, \;
\end{equation}
where $ x>0,\alpha, \lambda, c, \beta >0.$

\noindent
A random variable $X\thicksim  OGE-G(\Theta)$ has survival function in the form
\begin{equation} \label{5}
S(x)=1-\left\{ 1-e^{-\alpha \left[ e^{\frac{\lambda }{c}(e^{cx}-1)}-1\right] }\right\}^\beta.
\end{equation}

\noindent
The failure rate function of $OGE-G(\Theta $) is given by
\begin{equation} \label{6}
h(x)=\frac{f(x)}{S(x)}
=\frac{\alpha \beta \lambda e^{cx} e^{\frac{\lambda}{c}(e^{cx}-1)}e^{-\alpha \left[ e^{\frac{\lambda }{c}(e^{cx}-1)}-1\right]}
\left\{ 1-e^{-\alpha \left[ e^{\frac{\lambda }{c}(e^{cx}-1)}-1\right] }\right\}^{\beta -1}}{1-\left\{ 1-e^{-\alpha
\left[ e^{\frac{\lambda }{c}(e^{cx}-1)}-1\right]}\right\}^\beta}.
\end{equation}

\begin{center}
\includegraphics[width=0.44\textwidth]{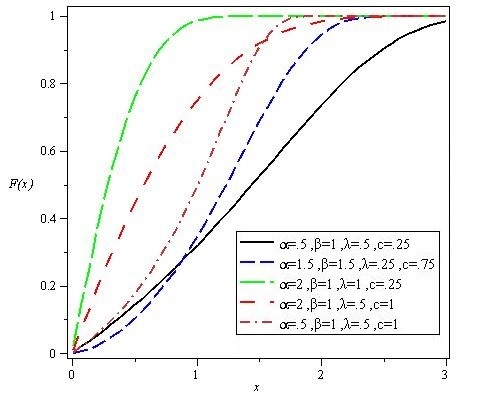}
\includegraphics[width=0.49\textwidth]{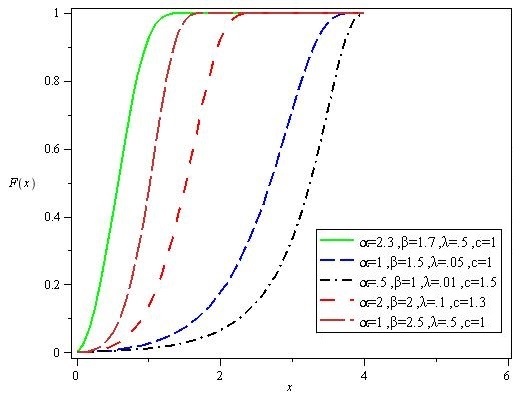} \\
Figure 1: The CDF of various OGE-G distributions for some values of the parameters.
\end{center}

\begin{center}
\includegraphics[width=0.49\textwidth]{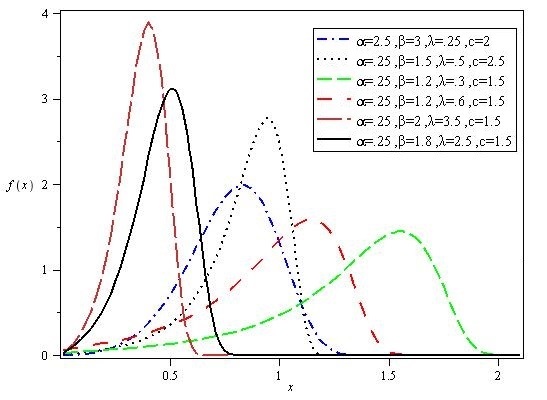}
\includegraphics[width=0.44\textwidth]{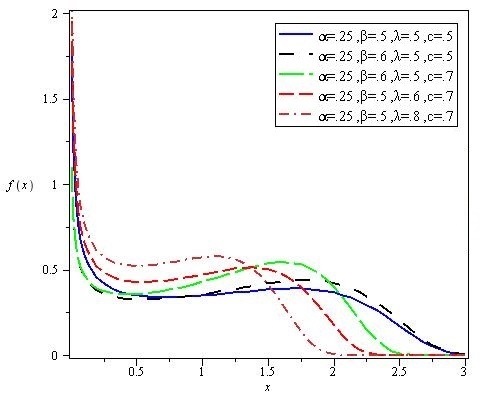}\\
Figure 2: The pdf's of various OGE-G distributions for some values of  the parameters.
\end{center}

\begin{center}
\includegraphics[width=0.6\textwidth]{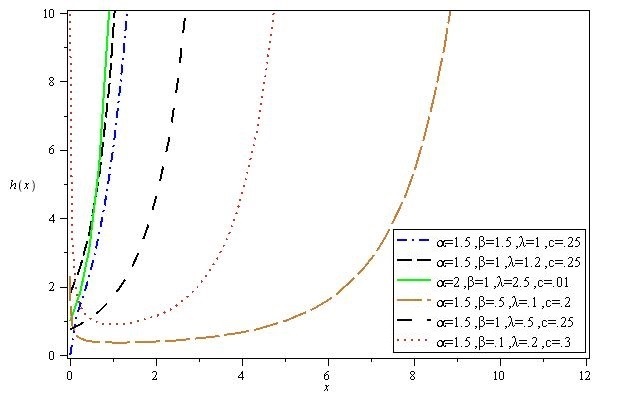}\\
Figure 3: The hazard of various OGE-G distributions for some values of the parameters.
\end{center}




\section{Statistical properties}
In this section, we study some statistical properties of OGE-G, especially quantile, median, mode and moments.

\subsection{Quantile and median of OGE-G}
The quantile of OGE-G($\Theta $) distribution is given by using
\begin{equation} \label{quan}
F(x_q)=q.
\end{equation}
Substituting from (\ref{3}) into (\ref{quan}), $x_q$ can be obtained as
\begin{equation} \label{7}
x_{q}=\frac{1}{c}\ln \left\{ 1+\frac{c}{\lambda }\ln \left[ 1-\frac{1}{\alpha }\ln (1-q^{\frac{1}{\beta }})\right] \right\}, \;0<q<1.
\end{equation}

\noindent
Setting $q=\frac{1}{2}$  in (\ref{7}),  we obtain the median of OGE-G($\alpha ,\lambda, c, \beta $) distribution as follows
\begin{equation} \label{8}
Med=\frac{1}{c}\ln \left\{ 1+\frac{c}{\lambda }\ln \left[ 1-\frac{1}{\alpha }\ln (1-(\frac{1}{2})^{\frac{1}{\beta }})\right] \right\}.
\end{equation}

\subsection{The mode of OGE-G}
In this subsection, we will derive the mode of the OGE-G($\Theta $) distribution by deriving its pdf with respect to $x$ and equal it to zero thus the mode of the OGE-G($\Theta $) distribution can be obtained as a nonnegative solution of the following nonlinear equation
\begin{eqnarray} \label{9}
&&
\left[ c+\lambda e^{cx}-\alpha \lambda e^{cx}e^{\frac{\lambda }{c}(e^{cx}-1)}\right] \left\{ 1-e^{-\alpha \left[ e^{\frac{\lambda }{c}(e^{cx}-1)}-1\right] }\right\} +\nonumber\\
&&
\; \; \; \; \; \; \; \; \; \;
\alpha \lambda (\beta -1)e^{cx}e^{\frac{\lambda }{c}(e^{cx}-1)}e^{-\alpha \left[ e^{\frac{\lambda }{c}(e^{cx}-1)}-1\right] }=0.
\end{eqnarray}%

\noindent
It is not possible to get an explicit solution of (\ref{9}) in the general case. Numerical methods should be used such as bisection or fixed-point method to solve it.

\subsection{The moments}
Moments are necessary and important in any statistical analysis, especially in applications. It can be used to study the most important features and characteristics of a distribution (e.g., tendency, dispersion, skewness and kurtosis). In this subsection, we will derive the rth moments of the OGE-G($\Theta $) distribution as infinite series expansion.

\begin{theorem} \label{theorem1}
If $X\sim OGE-G(\Theta)$, where $\Theta =(\alpha, \lambda, c, \beta)$, then the rth moment of $X$ is given by
\begin{equation*}
\mu_r^{'}=\sum_{i=0}^{\infty }\sum _{j=0}^{\infty } \sum_{k=0}^j \sum_{\ell=0}^{\infty } \sum_{m=0}^\ell (-1)^{i+j+k+m}
\begin{pmatrix} \beta -1 \\i \end{pmatrix}
\begin{pmatrix} \ell\\m \end{pmatrix}
\begin{pmatrix} j\\ k \end{pmatrix}
\frac{\alpha ^{j+1} \beta \lambda^{\ell+1}(i+1)^{j}(j-k+1)^{\ell} r!}{c^{\ell+r+1}j! \ell!(\ell-m+1)^{r+1}}.
\end{equation*}
\end{theorem}
\begin{proof}
The rth moment of the random variable $X$  with pdf $f(x)$ is defined by
\begin{equation} \label{10}
\mu_r ^{'}=\int_0^\infty x^{r}f(x)dx.
\end{equation}

\noindent
Substituting from (\ref{4}) into (\ref{10}), we get
\begin{equation} \label{11}
\mu_r ^{'}=\int_0^\infty x^{r}\alpha \beta \lambda e^{cx}e^{\frac{\lambda }{c}(e^{cx}-1)}e^{-\alpha \left[ e^{\frac{\lambda }{c}
(e^{cx}-1)}-1\right] }\left\{ 1-e^{\mathbf{-\alpha }\left[ \mathbf{e}^{\frac{\lambda }{c}(e^{cx}-1)}\mathbf{-1}\right] }\right\} ^{\beta -1}dx.
\end{equation}

\noindent
Since $0<1-e^{-\alpha \left[ e^{\frac{\lambda }{c}(e^{cx}-1)}-1\right] }<1$  for $x>0$, we have
\begin{equation} \label{12}
\left\{ 1-e^{\mathbf{-\alpha }\left[ e^{\frac{\lambda }{c}(e^{cx}-1)}-1\right] }\right\}^{\beta -1}=\sum_{i=0}^\infty \begin{pmatrix} \beta -1 \\i \end{pmatrix}
(-1)^{i}e^{-\alpha i\left[ e^{\frac{\lambda }{c}(e^{cx}-1)}-1\right] }.
\end{equation}

\noindent
Substituting from (\ref{12}) into (\ref{11}), we obtain
\begin{equation*}
\mu_r ^{'}=\sum_{i=0}^\infty \alpha \beta \lambda (-1)^{i} \begin{pmatrix} \beta -1 \\ i \end{pmatrix} \int_0^\infty
x^{r}e^{cx}e^{\frac{\lambda }{c}(e^{cx}-1)}e^{-\alpha (i+1)\left[ e^{\frac{\lambda }{c}(e^{cx}-1)}-1\right] }dx.
\end{equation*}

\noindent
Using the series expansion of $e^{-\alpha (i+1)(e^{\frac{\lambda }{c}(e^{cx}-1)}-1)}$, we get
\begin{equation*}
\mu_r^{'}=\sum_{i=0}^\infty \sum_{j=0}^\infty \alpha \beta \lambda (-1)^{i+j}\begin{pmatrix} \beta -1 \\i \end{pmatrix}
\frac{\alpha ^{j}(i+1)^{j}}{j!} \int_0^\infty x^{r}e^{cx}e^{\frac{\lambda }{c}(e^{cx}-1)}(e^{\frac{\lambda }{c}(e^{cx}-1)}-1)^{j}dx.
\end{equation*}
Using binomial expansion of $(e^{\frac{\lambda }{c}(e^{cx}-1)}-1)^{j}$, we get
\begin{equation*}
\mu_r ^{'}=\sum_{i=0}^\infty \sum_{j=0}^\infty \sum_{k=0}^j (-1)^{i+j+k} \begin{pmatrix} \beta -1 \\ i \end{pmatrix} \begin{pmatrix} j\\ k \end{pmatrix} \frac{\alpha ^{j+1}\beta \lambda (i+1)^{j}}{j!} \int_0^\infty x^{r} e^{cx}e^{\frac{\lambda }{c}(j-k+1)(e^{cx}-1)}dx.
\end{equation*}

\noindent
Using series expansion of $e^{\frac{\lambda }{c}(j-k+1)(e^{cx}-1)}$, we get
\begin{eqnarray*}
\mu_r^{'} &=&
\sum_{i=0}^\infty \sum_{j=0}^\infty \sum_{k=0}^j \sum_{\ell=0}^\infty (-1)^{i+j+k}
\begin{pmatrix} \beta -1 \\ i \end{pmatrix} \begin{pmatrix} j \\ k \end{pmatrix}
\frac{\alpha ^{j+1}\beta \lambda^{\ell+1}(i+1)^j(j-k+1)^{\ell}}{c^{\ell} j!\ell!} \times
\\
&&
\; \; \; \; \; \; \; \; \; \;
\int_0^\infty x^{r}e^{cx}(e^{cx}-1)^{\ell} dx,
\\
&=&
\sum_{i=0}^\infty \sum_{j=0}^\infty \sum_{k=0}^j \sum_{\ell=0}^\infty \sum_{m=0}^{\ell} (-1)^{i+j+k+m}
\begin{pmatrix} \beta -1 \\ i \end{pmatrix} \begin{pmatrix} \ell \\ m \end{pmatrix} \begin{pmatrix} j\\ k \end{pmatrix}
\frac{\alpha ^{j+1}\beta \lambda^{\ell+1}(i+1)^{j}(j-k+1)^{\ell}}{c^{\ell}j! \ell!} \times
\\
&&
\; \; \; \; \; \; \; \; \; \;
\int_0^\infty  x^{r}e^{c(\ell-m+1)x}dx.
\end{eqnarray*}

\noindent
By using the definition of gamma function in the form
\begin{equation*}
\Gamma (z)=x^{z} \int_0^\infty  e^{tx}t^{z-1}dt, \;z,x>0.
\end{equation*}

\noindent
Thus we obtain the moment of OGE-G($\Theta$ ) as follows
\begin{equation*}
\mu_r^{'}=\sum_{i=0}^\infty \sum_{j=0}^\infty \sum_{k=0}^j \sum_{\ell=0}^\infty \sum_{m=0}^ \ell (-1)^{i+j+k+m} \begin{pmatrix}
\beta -1\\ i\end{pmatrix} \begin{pmatrix} \ell \\ m \end{pmatrix} \begin{pmatrix} j\\k \end{pmatrix} \frac{\alpha ^{j+1}\beta \lambda^{\ell+1}(i+1)^{j}(j-k+1)^{\ell}r!}{c^{\ell+r+1}j! \ell!(\ell-m+1)^{r+1}}.
\end{equation*}
This completes the proof.
\end{proof}

\section{Order Statistic}
Let $X_1, X_2,\cdots, X_n$ be a simple random sample of size $n$ from OGE-G($\Theta$) distribution with cumulative distribution function $F(x;\Theta )$ and probability density function $f(x;\Theta)$ given by (\ref{3}) and (\ref{4}) respectively. Let $X_{1:n}\leq  X_{2:n}\leq \cdots \leq X_{n:n}$ denote the order statistics obtained from this sample. The probability density function of $X_{r:n}$ is given by
\begin{equation} \label{13}
f_{r:n}(x,\Theta )=\frac{1}{B(r,n-r+1)}\left[ F(x,\Theta )\right]^{r-1} \left[ 1-F(x,\Theta )\right]^{n-r} f(x;\Theta ),
\end{equation}

\noindent
where $f(x;\Theta ) $ and $ F(x,\Theta)$ are the pdf and  cdf of OGE-G($\Theta $) distribution given by (\ref{3}) and (\ref{4}) respectively and  $B(.,.)$ is the beta function. Since $0< F(x,\Theta )< 1 $ for $x >0$, we can use the binomial expansion of $\left[ 1-F(x,\Theta )\right] ^{n-r}$ given as follows
\begin{equation} \label{14}
\left[ 1-F(x,\Theta )\right]^{n-r}=\sum_{i=0}^{n-r} \begin{pmatrix} n-r \\ i \end{pmatrix} (-1)^i \left[ F(x,\Theta )\right]^i.
\end{equation}

\noindent
Substituting from (\ref{14}) into (\ref{13}), we have
\begin{equation} \label{15}
f_{r:n}(x,\Theta )=\frac{1}{B(r,n-r+1)}f(x;\Theta ) \sum_{i=0}^{n-r} \begin{pmatrix} n-r \\ i \end{pmatrix} (-1)^i
\left[ F(x,\Theta )\right]^{i+r-1}.
\end{equation}

\noindent
Substituting from (\ref{3}) and (\ref{4}) into (\ref{15}), we obtain
\begin{equation} \label{16}
f_{r:n}(x;\alpha ,\lambda ,c,\beta )=\sum_{i=0}^{n-r} \frac{(-1)^{i}n!}{i!(r-1)!(n-r-i)!(r+i)}f(x,\alpha ,\lambda ,c,(r+i)\beta )
\end{equation}

\noindent
Thus $f_{r:n}(x;\alpha ,\lambda ,c,\beta )$ defined in (\ref{16}) is the weighted average of the OGE-G distribution with different shape parameters.

\section{Estimation and Inference}
Now, we determine the maximum-likelihood estimators (MLE's) of the OGE-G parameters.

\subsection{Maximum likelihood estimators}
Let $X_1, X_2,\cdots, X_n$ be a random sample of size $n$ from OGE-G($ \Theta $), where $\Theta =(\alpha, \lambda, c, \beta )$, then the likelihood function $\mathcal{L}$ of this sample is defined as
\begin{equation} \label{17}
\mathcal{L}=\prod_{i=1}^n f(x_i;\alpha, \lambda, c, \beta ).
\end{equation}

\noindent
Substituting from (\ref{4}) into (\ref{15}), we get
\begin{equation*}
\mathcal{L}=\prod_{i=1}^n \alpha \beta \lambda e^{cx_{i}}e^{\frac{\lambda }{c}(e^{cx_{i}}-1)}e^{-\alpha \left[ e^{\frac{\lambda }{c}(e^{cx_{i}}-1)}-1\right] }\left\{ 1-e^{-\alpha \left[ e^{\frac{\lambda }{c}(e^{cx}-1)}-1\right] }\right\}^{\beta -1}.
\end{equation*}

\noindent
The log-likelihood function is given as follow
\begin{eqnarray} \label{18}
 L &= & n\ln (\alpha )+n\ln (\beta )+n\ln (\lambda )+ c \sum_{i=0}^n  x_{i}+ \frac{\lambda }{c} \sum_{i=0}^n (e^{cx_{i}}-1)- \alpha \sum_{i=0}^n(e^{\frac{\lambda }{c}(e^{cx_{i}}-1)}-1)\nonumber\\
&& +
(\beta -1) \sum_{i=0}^n \ln \left\{ 1-e^{-\alpha \left[ e^{\frac{\lambda }{c}(e^{cx}-1)}-1\right]}\right\}.
\end{eqnarray}

\noindent
The log-likelihood can be maximized either directly or by solving the nonlinear likelihood equations obtained by differentiating equation (\ref{18}) with respect to $\alpha, \lambda, c$ and $\beta $. The components of the score vector $U(\Theta )= \left( \frac{\partial L}{\partial \alpha },\frac{\partial L}{\partial \lambda },\frac{\partial L}{\partial c},\frac{\partial L}{\partial \beta } \right) $ are given by
\begin{eqnarray} \label{19}
\frac{\partial L}{\partial \beta } & = &
\frac{n}{\beta }+\sum_{i=1}^n \ln \left\{ 1-e^{-\alpha \left[ e^{\frac{\lambda }{c}(e^{cx}-1)}-1\right] }\right\},
\\ \label{20}
\frac{\partial L}{\partial \alpha } & = &
\frac{n}{\alpha }-\sum_{i=1}^n(e^{\frac{\lambda }{c}(e^{cx_{i}}-1)}-1)-(\beta -1)\sum_{i=1}^n \frac{
1-e^{\frac{\lambda }{c}(e^{cx_{i}}-1)}}{e^{\alpha \left[ e^{\frac{\lambda }{c}(e^{cx_{i}}-1)}-1\right]}-1 },
\\ \label{21}
\frac{\partial L}{\partial \lambda }& = &
\frac{n}{\lambda }+\frac{1}{c}\sum_{i=1}^n(e^{cx_{i}}-1)-\frac{\alpha }{c} \sum_{i=1}^n (e^{cx_{i}}-1)e^{\frac{\lambda }{c}(e^{cx_{i}}-1)}
+ \frac{\alpha (\beta -1)}{c} \sum_{i=1}^n \frac{\left( e^{cx_{i}}-1\right) e^{\frac{\lambda }{c}(e^{cx_{i}}-1)}}{e^{\alpha\left[ e^{\frac{\lambda }{c}(e^{cx}-1)}-1\right]}-1 },
\nonumber\\
&&
\\ \label{22}
\frac{\partial L}{\partial c} & = &
\sum_{i=1}^n x_{i}-\sum_{i=1}^n \tau(x_i,\lambda,c) - \alpha \sum_{i=1}^n \tau(x_i,\lambda,c) e^{\frac{\lambda }{c} (e^{cx_{i}}-1)}  +
(\beta -1)\alpha \sum_{i=1}^n \frac{\tau(x_i,\lambda,c) e^{\frac{\lambda }{c}(e^{cx_{i}}-1)}}{e^{\alpha \left[ e^{\frac{\lambda }{c}(e^{cx}-1)}-1\right] }-1},
\nonumber\\
&&
\end{eqnarray}
where
$ \tau(x_i, \lambda,c)=-\frac{\lambda }{c^2} (e^{cx_{i}}-1)+\frac{\lambda}{c} x_{i}e^{cx_{i}}.$

\noindent
The normal equations can be obtained by setting the above non-linear equations (\ref{19})-(\ref{22}) to zero. That is, the normal equations take the following form
\begin{eqnarray} \label{23}
&&
\frac{n}{\beta }+\sum_{i=1}^n \ln \left\{ 1-e^{-\alpha \left[ e^{\frac{\lambda }{c}(e^{cx}-1)}-1\right] }\right\} =  0,
\\ \label{24}
&&
\frac{n}{\alpha }-\sum_{i=1}^n(e^{\frac{\lambda }{c}(e^{cx_{i}}-1)}-1)-(\beta -1)\sum_{i=1}^n \frac{
1-e^{\frac{\lambda }{c}(e^{cx_{i}}-1)}}{e^{\alpha \left[ e^{\frac{\lambda }{c}(e^{cx_{i}}-1)}-1\right]}-1 }  = 0,
\\ \label{25}
&&
\frac{n}{\lambda }+\frac{1}{c}\sum_{i=1}^n(e^{cx_{i}}-1)-\frac{\alpha }{c} \sum_{i=1}^n (e^{cx_{i}}-1)e^{\frac{\lambda }{c}(e^{cx_{i}}-1)}
+ \frac{\alpha (\beta -1)}{c} \sum_{i=1}^n \frac{\left( e^{cx_{i}}-1\right) e^{\frac{\lambda }{c}(e^{cx_{i}}-1)}}{e^{\alpha\left[ e^{\frac{\lambda }{c}(e^{cx}-1)}-1\right]}-1 } =  0,
\nonumber\\
&&
\\ \label{26}
&&
\sum_{i=1}^n x_{i}-\sum_{i=1}^n \tau(x_i,\lambda,c) - \alpha \sum_{i=1}^n \tau(x_i,\lambda,c) e^{\frac{\lambda }{c} (e^{cx_{i}}-1)}  +
(\beta -1)\alpha \sum_{i=1}^n \frac{\tau(x_i,\lambda,c) e^{\frac{\lambda }{c}(e^{cx_{i}}-1)}}{e^{\alpha \left[ e^{\frac{\lambda }{c}(e^{cx}-1)}-1\right] }-1} =  0.
\nonumber\\
&&
\end{eqnarray}

\noindent
The normal equations do not have explicit solutions and they have to be obtained numerically. From equation (\ref{23}) the MLEs of $\beta$  can be obtained as follows
\begin{equation} \label{27}
\hat{\beta}=\frac{-n}{\sum_{i=1}^n \ln \left\{ 1-e^{-\alpha \left[ e^{\frac{\lambda }{c}(e^{cx}-1)}-1\right] }\right\} }.
\end{equation}

\noindent
Substituting from (\ref{27}) into (\ref{24}), (\ref{25}), and (\ref{26}), we get the MLEs of $\alpha, \lambda, c$ by solving the following system of non-linear equations
\begin{eqnarray} \label{28}
&&
\frac{n}{\hat{\alpha}}-\sum_{i=1}^n(e^{\frac{\hat{\lambda}}{\hat{c}}(e^{\hat{c}x_{i}}-1)}-1)-(\hat{\beta}-1)\sum_{i=1}^n \frac{ 1-e^{\frac{\hat{\lambda}}{\hat{c}}(e^{\hat{c}x_{i}}-1)}}{e^{\hat{\alpha}\left[ e^{\frac{\hat{\lambda}}{\hat{c}}(e^{\hat{c}x_{i}}-1)}-1\right] }-1}=0,
\\ \label{29}
&&
\frac{n}{\hat{\lambda}}+\frac{1}{\hat{c}} \sum_{i=1}^n(e^{\frac{\hat{\lambda}}{\hat{c}}(e^{\hat{c}x_{i}}-1)}-1)- \frac{\hat{\alpha}}{\hat{c}} \sum_{i=1}^n( e^{\hat{c} x_{i}}-1)e^{\frac{\hat{\lambda}}{\hat{c}}(e^{\hat{c}x_{i}}-1)}
+\frac{\hat{\alpha}(\hat{\beta}-1)}{\hat{c}} \sum_{i=1}^n \frac{(  e^{\hat{c}x_{i}}-1) e^{\frac{\hat{\lambda}}{\hat{c}}(e^{\hat{c}x_{i}}-1)} }{e^{\hat{\alpha}\left[ e^{\frac{\hat{\lambda} }{\hat{c}}(e^{\hat{c}x_{i}}-1)}-1\right] }-1}=0,
\nonumber\\
&&
\\ \label{30}
&& \sum_{i=1}^n x_{i}- \sum_{i=1}^n \tau(x_i,\hat{\lambda}, \hat{c} )
-\hat{\alpha} \sum_{i=1}^n\tau(x_i,\hat{\lambda}, \hat{c} ) e^{\frac{\hat{\lambda}}{\hat{c}}(e^{\hat{c} x_{i}}-1)} +(\hat{\beta}-1)\hat{\alpha}\sum_{i=1}^n \frac{\tau(x_i, \hat{\lambda}, \hat{c})
e^{\frac{\hat{\lambda}}{\hat{c}}(e^{\hat{c}x_{i}}-1)}}{e^{\hat{\alpha}\left[ e^{\frac{\hat{\lambda}}{\hat{c}}(e^{\hat{c}x_{i}}-1)}-1\right] }-1}=0,
\nonumber\\
&&
\end{eqnarray}
where $\tau(x_i, \hat{\lambda}, \hat{c})= -\tfrac{\hat{\lambda}}{\hat{c}^{2}}({\small e}^{\hat{c}%
x_{i}}-1)+\tfrac{\hat{\lambda}}{\hat{c}}x_{i}{\small e}^{\hat{c}x_{i}}.$

\noindent
These equations cannot be solved analytically and statistical software can be used to solve the equations numerically. We can use iterative techniques such as Newton Raphson type algorithm to obtain the estimate $\hat{\beta}$.

\subsection{Asymptotic confidence bounds}
In this subsection, we derive the asymptotic confidence intervals of the unknown parameters $\alpha, \lambda, c, \beta $ when $\alpha>0, \lambda >0, c>0 $ and $\beta >0$. The simplest large sample approach is to assume that the MLEs($\alpha, \lambda, c, \beta $) are approximately multivariate normal with mean ($\alpha, \lambda, c, \beta $) and covariance matrix $I_{0}^{-1}$, where $I_{0}^{-1}$  is the inverse of the observed information matrix which defined as follows
\begin{eqnarray} \label{31}
I_{0}^{-1} & = &
-\left[
\begin{array}{cccc}
\frac{\partial^2 L}{\partial \alpha ^2} & \frac{\partial^2 L}{ \partial \lambda \partial \alpha } & \frac{\partial^2 L}{ \partial c \partial \alpha} & \frac{\partial^2 L}{ \partial \beta \partial \alpha} \\
\frac{\partial^2 L}{ \partial \alpha \partial \lambda} & \frac{\partial^2 L }{\partial \lambda^2 } & \frac{\partial^2 L}{ \partial c \partial \lambda} & \frac{\partial^2 L}{ \partial \beta \partial \lambda} \\
\frac{\partial^2 L}{\partial \alpha \partial c} & \frac{\partial^2 L}{ \partial \lambda \partial c} & \frac{\partial^2 L}{\partial c^2} & \frac{\partial^2}{\partial \beta \partial c } \\
\frac{\partial^2}{ \partial \alpha \partial \beta} & \frac{\partial^2}{ \partial \lambda \partial \beta} & \frac{\partial^2 L}{ \partial c \partial \beta} & \frac{\partial^2}{\partial \beta^2}
\end{array}%
\right] ^{-1}
\nonumber\\
&= &
\left[
\begin{array}{cccc}
var(\hat{\alpha}) & cov(\hat{\lambda}, \hat{\alpha}) & cov(\hat{c}, \hat{\alpha}) & cov(\hat{\beta}, \hat{\alpha}) \\
cov(\hat{\alpha}, \hat{\lambda}) & var(\hat{\lambda}) & cov(\hat{c}, \hat{\lambda}) & cov(\hat{\beta}, \hat{\lambda}) \\
cov(\hat{\alpha}, \hat{c}) & cov(\hat{\lambda}, \hat{c}) & var(\hat{c}) & cov(\hat{\beta}, \hat{c}) \\
cov(\hat{\alpha}, \hat{\beta}) & cov(\hat{\lambda}, \hat{\beta}) & cov(\hat{c}, \hat{\beta}) & var(\hat{\beta})
\end{array}%
\right] .
\end{eqnarray}

\noindent
The second partial derivatives included in $I_{0}^{-1}$ are given as follows
\begin{eqnarray*}
\frac{\partial^2 L}{\partial \beta ^{2}} & = &
\frac{-n}{\beta^2},
\\
\frac{\partial^2 L}{ \partial \alpha \partial \beta } & = & \sum_{i=1}^n \frac{\left( e^{\frac{\lambda }{c}(e^{cx_{i}}-1)}-1\right) e^{-\alpha \left[ e^{\frac{ \lambda }{c}(e^{cx_{i}}-1)}-1\right] }}{1-e^{-\alpha \left[e^{\frac{\lambda }{c}(e^{cx_{i}}-1)}-1\right] }},
\\
\frac{\partial^2 L}{ \partial \lambda \partial \beta } &= &
\frac{\alpha }{c} \sum_{i=1}^n \frac{\left( e^{cx_{i}}-1\right) A_{i}}{1-e^{-\alpha\left[ e^{\frac{\lambda}{c}(e^{cx_{i}}-1)}-1\right] }},
\\
\frac{\partial^2 L}{ \partial c \partial \beta} & = &
\alpha \sum_{i=1}^n \frac{A_{i}B_{i}}{1-e^{-\alpha \left[ e^{\frac{\lambda }{c}(e^{cx_{i}}-1)}-1\right] }},
\\
\frac{\partial^2 L}{\partial \alpha^2} & = &
\frac{-n}{ \alpha ^{2}}-(\beta -1) \sum_{i=1}^n\frac{\left( e^{\frac{\lambda }{c}(e^{cx_{i}}-1)}-1\right) ^{2}e^{-\alpha \left[
e^{\frac{\lambda }{c}(e^{cx_{i}}-1)}-1\right] }}{\left\{ 1-e^{-\alpha\left[ e^{\frac{\lambda }{c}(e^{cx}-1)} -1\right] }\right\}^2},
\\
\frac{\partial^2 L}{ \partial \lambda \partial \alpha } & =&
\frac{-1}{c}\sum_{i=1}^n \left( e^{cx_{i}}-1\right) e^{\frac{\lambda }{c}(e^{cx_{i}}-1)}+\frac{(\beta-1)}{c} \sum_{i=1}^n \frac{\left( e^{cx_{i}}-1\right) A_{i}C_{i}}{\left\{ 1-e^{-\alpha\left[ e^{\frac{\lambda }{c}(e^{cx}-1)}-1\right] }\right\}^2},
\\
\frac{\partial^2 L}{ \partial c \partial \alpha} &= &-
\sum_{i=1}^n B_{i}e^{\frac{\lambda }{c}(e^{cx_{i}}-1)}+(\beta -1)\sum_{i=1}^n \frac{A_{i}B_{i}C_{i}}{\left\{ 1-e^{-\alpha\left[ e^{\frac{\lambda }{c}(e^{cx}-1)}-1\right] }\right\}^2},
\\
\frac{\partial^2 L}{\partial \lambda^2} & = &
\frac{-n}{\lambda ^{2}}-\frac{\alpha }{c^{2}} \sum_{i=1}^n \left( e^{cx_{i}}-1\right) ^{2}e^{\frac{\lambda }{c}(e^{cx_{i}}-1)} + \frac{\alpha {\small (\beta -1)}}{c^{2}} \sum_{i=1}^n \frac{\left( e^{cx_{i}}-1\right) ^{2}A_{i}D_{i}}{\left\{ 1-e^{-\alpha \left[ e^{\frac{\lambda }{c}(e^{cx}-1)}-1\right]}\right\}^2},
\\
\frac{\partial^2 L}{ \partial c \partial \lambda } &= &
\frac{1}{\lambda }\sum_{i=1}^n B_i- \frac{\alpha}{c} \sum_{i=1}^n \left[ \frac{-1}{c}(e^{cx_i}-1)+ x_i e^{cx_i}+ B_i \left( e^{cx_i}-1 \right) \right] e^{\frac{\lambda }{c}(e^{cx_i}-1)}
\nonumber\\
&&
+
\frac{(\beta -1)\alpha }{\lambda } \sum_{i=1}^n \frac{\left[ 1-e^{-\alpha \left[ e^{\frac{\lambda}{c}(e^{cx_{i}}-1)} -1\right]} + \frac{\lambda }{c}(e^{cx_{i}}-1)D_{i}\right] A_i B_i}{\left\{ 1-e^{-\alpha \left[ e^{\frac{\lambda }{c}(e^{cx}-1)}-1\right]} \right\}^2},
\end{eqnarray*}

\begin{eqnarray*}
\frac{\partial^2 L}{\partial c^2}& = &\sum_{i=1}^n E_i-
\alpha \sum_{i=1}^n \left[ E_i+B_i^2 \right] e^{\frac{\lambda }{c}(e^{cx_{i}}-1)} +
\; \; \; \; \; \; \; \; \; \; \; \; \; \; \; \; \; \; \; \; \; \; \; \; \; \; \; \; \; \; \; \; \; \; \; \; \; \; \; \; \; \; \; \; \; \; \; \; \; \; \; \; \; \;
\\
&&
(\beta -1)\alpha \sum_{i=1}^n \frac{\left[ B_i^2 D_i + E_i \left( 1-e^{-\alpha \left[ e^{\frac{\lambda }{c}(e^{cx_{i}}-1)} -1 \right] }\right) \right] A_i}{\left\{1-e^{-\alpha\left[ e^{\frac{\lambda }{c}(e^{cx}-1)}-1\right] }\right\}^2},
\end{eqnarray*}

\noindent
where $ A_i  =  e^{\frac{\lambda }{c}(e^{cx_{i}}-1)}e^{-\alpha \left[ e^{\frac{\lambda }{c}(e^{cx}-1)}-1\right] },
\qquad B_i  =  -\frac{\lambda }{c^{2}}(e^{cx_{i}}-1)+\frac{\lambda }{c} (x_{i}e^{cx_{i}}),$
$$
C_i  =  1-\alpha (e^{\frac{\lambda }{c}(e^{cx_{i}}-1)}-1)-e^{-\alpha \left[ e^{\frac{\lambda }{c}(e^{cx}-1)}-1\right] },
\qquad
D_i  =  1-\alpha e^{\frac{\lambda }{c}(e^{cx_{i}}-1)}-e^{-\alpha \left[ e^{\frac{\lambda }{c}(e^{cx}-1)}-1\right] },
$$
$$
E_i  =  \frac{2\lambda }{c^{3}}\left( e^{cx_{i}} -1\right) -\frac{2\lambda }{c^{2}}x_{i}e^{cx_{i}}+\frac{\lambda }{c} x_{i}^{2}e^{cx_{i}}.
$$

\noindent
The asymptotic $(1-\gamma )100\%$ confidence intervals of $\alpha, \lambda, c $ and  $\beta $ are
$\hat{\alpha}\pm z_{\frac{\gamma }{2}}\sqrt{var(\hat{\alpha})}, \; \hat{\lambda}\pm z_{\frac{\gamma }{2}}\sqrt{var(\hat{\lambda})}, \; \hat{c}\pm z_{\frac{\gamma }{2}}\sqrt{var(\hat{c})} $ and  $\hat{\beta}\pm z_{\frac{%
\gamma }{2}}\sqrt{var(\hat{\beta})})$ respectively, where $z_{\frac{\gamma }{2}}$ is the upper ( $\frac{\gamma }{2}$)th percentile of the standard normal distribution.

\section{Data Analysis}
In this section we perform an application to real data to illustrate that the OGE-G can be a good lifetime model, comparing with many known distributions such as the Exponential, Generalized Exponential, Gompertz, Generalized Gompertz and Beta Gompertz distributions (ED, GE, G, GG, BG), see \cite{7,6,3,8}.

\noindent
Consider the data have been obtained from Aarset \cite{1}, and widely reported in some literatures, see  \cite{2,3,9,12,13}. It represents the lifetimes of 50 devices, and also, possess a bathtub-shaped failure rate property, Table 1.

\begin{center}
{Table 1: The data from Aarset \cite{1}.}
\begin{tabular}{rrrrrrrrrrrrrrrrrr} \hline
0.1  & 0.2 & 1 & 1 & 1 & 1 & 1 & 2 & 3 & 6 & 7 & 11 & 12 & 18 & 18 & 18 & 18  \\
18 & 21 & 32 & 36 & 40 & 45 & 46 & 47 & 50 & 55 & 60 & 63 & 63 & 67 & 67 & 67 & 67  \\
72 & 75 & 79 & 82 & 82 & 83 & 84 & 84 & 84 & 85 & 85 & 85 & 85 & 85 & 86 & 86 & \\
\hline
\end{tabular}
\end{center}

\noindent
Based on some goodness-of-fit measures, the performance of the OGE-G distribution is compared with others five distributions: E, GE, G, GG, and BG distributions. The MLE's of the unknown parameters for these distributions are given in Tables 2 and 3. Also, the values of the log-likelihood functions (-$L$),  the statistics K-S (Kolmogorov-Smirnov), AIC (Akaike Information Criterion), the statistics AICC (Akaike Information Citerion with correction) and BIC (Bayesian Information Criterion) are calculated for the six distributions in order to verify which distribution fits better to these data.

\newpage
\begin{center}
{Table 2: MLE's of parameters for Aarset data.}\\
\begin{tabular}{cccccc} \hline
\hspace{0.5cm} Model \hspace{0.5cm}  & \hspace{0.5cm} $\hat{\alpha}$ \hspace{0.5cm} & \hspace{0.5cm} $\hat{\beta}$ \hspace{0.5cm} & \hspace{0.5cm} $\hat{\lambda}$  \hspace{0.5cm}& \hspace{0.5cm}$\hat{c}$  \hspace{0.5cm} & \hspace{0.5cm} -L  \hspace{0.5cm}   \\ \hline
E     & 0.0219   &    --    &   --      &  --      & 241.0896 \\
GE    & 0.0212   & 0.9012   &     --    &   --     & 240.3855 \\
G     &    --    &     --   & 0.00970   & 0.0203   & 235.3308 \\
GG    &    --    & 0.2625   & 0.00010  & 0.0828    & 222.2441 \\
BG    & 0.2158   &  0.2467  & 0.00030   & 0.0882   & 220.6714 \\
OGE-G & 0.0400   & 0.1940   & 0.000345  & 0.0780   & 215.9735 \\ \hline
\end{tabular}
\end{center}

\begin{center}
{Table 3: The AIC, CAIC, BIC and K-S values for Aarset data.}
\begin{tabular}{cccccc} \hline
\hspace{0.3cm} Model \hspace{0.3cm} & \hspace{0.3cm} AIC \hspace{0.3cm} & \hspace{0.3cm} AICC \hspace{0.3cm}& \hspace{0.3cm} BIC \hspace{0.3cm}  & \hspace{0.3cm} K-S \hspace{0.3cm}& \hspace{0.3cm} p-value(K-S) \\ \hline
E     & 484.1792 & 484.2625 & 486.0912 & 0.19110 & 0.0519 \\
GE    & 484.7710 & 485.0264 & 488.5951 & 0.19400 & 0.0514 \\
G     & 474.6617 & 475.1834 & 482.3977 & 0.16960 & 0.1123 \\
GG    & 450.4881 & 451.0099 & 456.2242 & 0.14090 & 0.2739 \\
BG    & 449.3437 & 450.2326 & 456.9918 & 0.13220 & 0.3456 \\
OGE-G & 423.9470 & 424.8359 & 447.5951 & 0.13205 & 0.3476 \\ \hline
\end{tabular}
\end{center}

\noindent
Based on Tables 2 and 3, it is shown that OGE-G$ (\alpha,\lambda,c,\beta)$  model is the best among of those distributions because it has the smallest value of (K-S), AIC, CAIC  and BIC test.

\noindent
Substituting the MLE's of the unknown parameters $\alpha, \lambda, c, \beta $  into (\ref{31}), we get estimation of the variance covariance matrix as the following

$$
I_0^{-1}=\left[
\begin{array}{rrrr}
1.5022\times 10^{-3} & 2.7638\times 10^{-7} & -1.4405\times 10^{-4} & 9.6331\times 10^{-4} \\
2.7638\times 10^{-7} & 4.4773\times 10^{-8} & -1.7837\times 10^{-6} & 2.785\,8\times 10^{-6} \\
-1.4405\times 10^{-4} & -1.7837\times 10^{-6} & 8.4428\times 10^{-5} & -1.9127\times 10^{-4} \\
9.6331\times 10^{-4} & 2.7858\times 10^{-6} & -1.9127\times 10^{-4} & 1.5309\times 10^{-3}
\end{array}
\right]
$$

\noindent
The approximate 95\% two sided confidence intervals of the unknown parameters $\alpha, \lambda, c$  and $\beta$ are $\left[ 0,0.116 \right], \;  \left[ 0,0.001\right], \;  \left[ 0.060,0.096\right], \;  \left[ 0.116,0.270\right]$, respectively.\\

\noindent
To show that the likelihood equation have unique solution, we plot the profiles of the log-likelihood function of $\alpha, \lambda, \beta$ and $c$ in Figures 4-5.

\begin{center}
\includegraphics[width=0.43\textwidth]{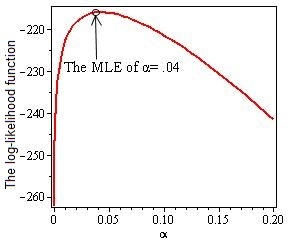}
\includegraphics[width=0.45\textwidth]{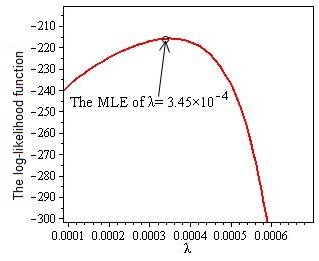}\\
Figure 4: The profile of the log-likelihood function of $\alpha, \lambda$.
\end{center}

\begin{center}
\includegraphics[width=0.37\textwidth]{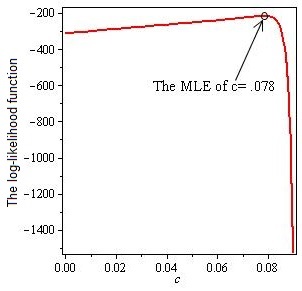}
\includegraphics[width=0.45\textwidth]{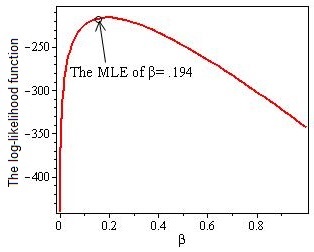}\\
Figure 5: The profile of the log-likelihood function of $c, \beta$.
\end{center}



\noindent
The nonparametric estimate of the survival function using the Kaplan-Meier method and its fitted parametric estimations when the distribution is assumed to be $ED, GED, GD, GGD$ and $OGE-GD$ are computed and plotted in Figure 6.

\begin{center}
\includegraphics[width=0.45\textwidth]{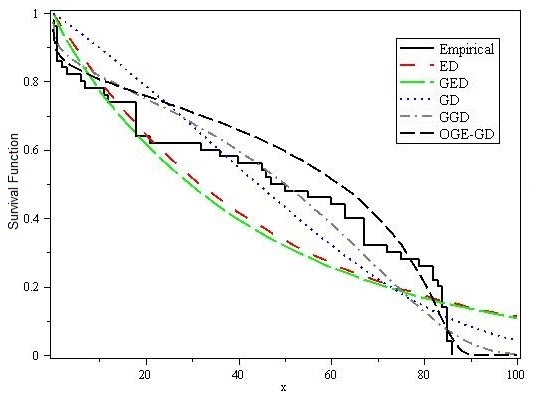}\\
Figure 6: The Kaplan-Meier estimate of the survival function.
\end{center}


\noindent
Figure 7, gives the form of the hazard rate for the $ED, GED, GD, GGD, BGD$ and $OGE-GD$ which are used to fit the data after replacing the unknown parameters included in each distribution by their MLE.

\begin{center}
\includegraphics[width=0.45\textwidth]{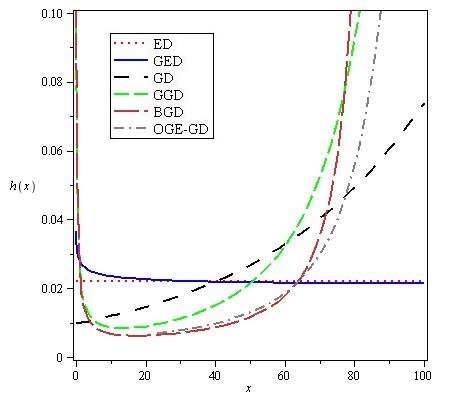}\\
Figure 7: The failure rate function for the data.
\end{center}


\section{Conclusions}
In this paper, we propose a new model, called the Odd Generalized Exponential-Gompertz (OGE-G) distribution and studied its different properties. Some statistical properties of this distribution have been derived and discussed. The quantile, median, mode and moments of OGE-G are derived in closed forms. The distribution of the order statistics are discussed. Both point and asymptotic confidence interval estimates of the parameters are derived using the maximum likelihood method and we obtained the observed Fisher information matrix. We use application on set of real data to compare the OGE-G with other known distributions such as
Exponential (E), Generalized Exponential (GE), Gompertz (G), Generalized Gompertz (GG) and Beta-Gompertz (BG). Applications on set of real data showed that the OGE-G is the best distribution for fitting these data sets compared with other distributions considered in this article.

\end{document}